\pdfoutput=1
\RequirePackage{ifpdf}
\ifpdf 
\documentclass[pdftex]{sigma}
\else
\documentclass{sigma}
\fi

\numberwithin{equation}{section}

\newtheorem{Theorem}{Theorem}[section]
\newtheorem{Corollary}[Theorem]{Corollary}
\newtheorem{Proposition}[Theorem]{Proposition}
{ \theoremstyle{definition}
\newtheorem{Remark}[Theorem]{Remark} }

\newcommand{\de}{\delta}
\newcommand{\ep}{\varepsilon}

\newcommand{\CC}{{\mathbb{C}}}

\newcommand{\HH}{{\mathbb{H}}}

\newcommand{\PP}{{\mathbb{P}}}

\newcommand{\ZZ}{{\mathbb{Z}}}

\newcommand{\NN}{{\mathbb{N}}}

\def\ZZ{{\mathbb Z}}

\def\t{\theta}
\def\s{\sigma}
\def\r{\rho}
\def\T{\Theta}

\def\tt#1#2{{\t\left[\begin{matrix}{#1}\\ {#2}\end{matrix}\right]}}

\begin{document}

\allowdisplaybreaks

\newcommand{\arXivNumber}{2004.05099}

\renewcommand{\PaperNumber}{057}

\FirstPageHeading

\ShortArticleName{On Frobenius' Theta Formula}

\ArticleName{On Frobenius' Theta Formula}

\Author{Alessio FIORENTINO and Riccardo SALVATI MANNI}
\AuthorNameForHeading{A.~Fiorentino and R.~Salvati Manni}

\Address{Sapienza Universit\`{a} di Roma, Italy}
\Email{\href{mailto:alessio.fiorentino@uniroma1.it}{alessio.fiorentino@uniroma1.it}, \href{mailto:salvati@mat.uniroma1.it}{salvati@mat.uniroma1.it}}

\ArticleDates{Received April 14, 2020, in final form June 11, 2020; Published online June 17, 2020}

\Abstract{Mumford's well-known characterization of the hyperelliptic locus of the mo\-duli space of ppavs in terms of vanishing and non-vanishing theta constants is based on Neumann's dynamical system. Poor's approach to the characterization uses the cross ratio. A~key tool in both methods is Frobenius' theta formula, which follows from Riemann's theta formula. In a 2004 paper Grushevsky gives a different characterization in terms of cubic equations in second order theta functions. In this note we first show the connection between the methods by proving that Grushevsky's cubic equations are strictly related to Frobenius' theta formula and we then give a new proof of Mumford's characterization via Gunning's multisecant formula.}


\Classification{14H42; 14H45; 14K25; 14K12; 14H40}

\section{Introduction}

We denote by $\HH_g$ the {\it Siegel upper half-space}~-- the
space of complex symmetric $g\times g$ matrices with positive definite
imaginary part. An element $\tau\in\HH_g$ is called a {\it period
matrix} and defines the complex abelian variety $X_\tau:=\CC^g/\ZZ^g+\tau
\ZZ^g$. It is a well-known fact that the moduli space of principally polarized abelian varieties (ppavs for short) can be identified with the quotient of $\HH_g$ by an action of the symplectic group ${\rm Sp}(2g, \ZZ)$ which is a generalization of the standard action of ${\rm SL}(2,\ZZ)$ on the complex upper-half plane, namely,
\[
\gamma \cdot \tau := (a \tau + b) \cdot (c \tau + d)^{-1},\qquad \forall\, \gamma = \begin{pmatrix} a & b \\ c & d \end{pmatrix} \in {\rm Sp}(2g,\ZZ),
\]
where $a$, $b$, $c$, $d$ are $g \times g$ blocks. A ppav is called {\it irreducible} if it is not isomorphic to a product of two lower-dimensional ppavs.

For $\ep,\de\in \ZZ_2^g$ and $z\in \CC^g$ we define the {\it first order theta function with characteristic $[\ep,\de]$} to be
\[
\tt\ep\de(\tau,z):=\sum\limits_{m\in\ZZ^g} \exp \pi {\rm i} \left[\left(
m+\frac{\ep}{2}\right)^t\tau \left(m+\frac{\ep}{2}\right)+2\left(m+\frac{\ep}{2}\right)^t\left(z+
\frac{\de}{2}\right)\right].
\]

Characteristics can be defined for $\ep$, $\de$ in $\ZZ^g$, but the {\it reduction formula}
\begin{gather*}
\tt{\ep+2\ep'}{\de+2\de'} (\tau , z) = (-1)^{\langle \ep ,\de'\rangle} \tt\ep\de(\tau,z) \end{gather*}
(where the symbol $\langle \cdot , \cdot \rangle$ stands for the standard inner product) shows that these functions are uniquely determined up to a sign by considering~$\ep$ and~$\de$ as vectors of zeros and ones. Henceforward we shall work with reduced characteristics with the agreement that a theta function associated with a sum of characteristics is meant to be non-reduced.

Characteristics are defined {\it even} or {\it odd} depending on whether respectively $\langle \ep , \de\rangle =0,\, =1$ ${\rm mod}\, 2$. A straightforward computation shows there are $ 2^{g-1}\big( 2^{g}+1\big)$ even characteristics and $2^{g-1}\big( 2^{g}-1\big)$ odd ones. It is easily seen that theta functions associated with even characteristics are even functions in the variable~$z$, whereas those associated with odd characteristics are odd functions.
Triplets of characteristics $[\ep,\de]$, $[\ep',\de']$, $[\ep'',\de'']$ are also called {\it azygetic} if $\langle \ep , \de\rangle+\langle\ep' , \de'\rangle+\langle\ep'' , \de''\rangle+\langle\ep + \ep' + \ep'' , \de + \de' + \de''\rangle=1$. More generally, a $k$-tuple of characteristics is called azygetic if its subtriplets are azygetic.

Theta functions satisfy an addition formula (cf.~\cite{Ig72} for a general formulation and~\cite{Gr} or~\cite{GSM} for the specific version we use here)
\[
\theta \begin{bmatrix} \ep \\ \de \end{bmatrix} (2\tau, 2z) \theta \begin{bmatrix} \ep + \ep' \\ \de \end{bmatrix} (2\tau, 2w)= \frac{1}{2^g} \sum_{\s \in \ZZ_2^{g}}(-1)^{\langle \ep , \s\rangle} \theta \begin{bmatrix} \ep '\\ \de +\s \end{bmatrix} (\tau, z+w) \theta \begin{bmatrix} \ep'\\ \s \end{bmatrix} (\tau,z-w).
\]

For a given $\tau \in H_g$ we denote by $X_{\tau}[2]$ the group of points of order two on $X_{\tau}$ and by $L_{\tau}$ a~symmetric line bundle on $X_\tau$ defining the principal polarization. Note that
\[
\tt\ep\de(\tau,z)=\exp \pi {\rm i} \left[\frac{\ep^t}{2}\tau\frac{\ep}{2} +\frac{\ep}{2}\left(z+\frac{\de}{2}\right) \right]\tt{0}{0}\left( \tau,z+ \tau\frac{\ep}{2} +\frac{\de}{2} \right)
\]
and the function $z \to
\theta\left[\begin{smallmatrix}\ep \\ \de \end{smallmatrix}\right] (\tau,z)$ defines the unique (up to scalar multiplication)
section of $t_x^*L_{\tau}$, i.e., the translate of the line bundle $L_{\tau}$ by the point of order two $x=(\tau \ep +\de)/2$. Because of the duplication map, the functions $z \rightarrow
\theta\left[\begin{smallmatrix}\ep \\ \de \end{smallmatrix}\right] (\tau,2z)$ are a basis of $H^0\big(X_\tau, L_{\tau}^{\otimes 4}\big)$.

For $\ep\in\ZZ_2^g$ we also
define the {\it second order theta function with characteristic $\ep$} to
be
\[
\T[\ep](\tau,z):=\tt{\ep}{0}(2\tau,2z).
\]
The functions $z \rightarrow \T[\ep](\tau,z)$ are a basis of $H^0\big(X_\tau, L_{\tau}^{\otimes 2}\big)$. Since these functions are even, i.e., $
\T[\ep](\tau,-z)=
\T[\ep](\tau,z)$, they induce a map
\[
{\rm Th}_2\colon \ X_\tau\to \PP^{2^g-1}
\]
 that factorizes along the Kummer variety $ K_{\tau}:=X_\tau/\pm 1$.

By evaluating first order theta functions at $z=0$ we get the so-called {\it theta constants} $\theta\left[\begin{smallmatrix}\ep \\ \de \end{smallmatrix}\right] (\tau,0)$; because of the parity of theta functions, theta constants associated with odd characteristics are clearly trivial. Analogously, second order theta functions evaluated at $z=0$ yield {\it second order theta constants} $\T[\ep](\tau,0)$. As functions of $\tau$ these are not well defined on the moduli space of ppavs but, as a consequence of how they transform under the action of ${\rm Sp}(2g, \ZZ)$, they induce maps on finite covers of the moduli space. More precisely, if we introduce the following family of subgroups of finite index in ${\rm Sp}(2g, \ZZ)$
\begin{gather*}
\Gamma_g[n]:= \operatorname{Ker}({\rm Sp}(2g,\ZZ) \to {\rm Sp}(2g,\ZZ_n)), \qquad n \in \NN,\\
\Gamma_g[n,2n]:=\left\lbrace\gamma = \begin{pmatrix} a & b \\ c & d \end{pmatrix} \in \Gamma_g[n] \,|\, {\rm diag}\big(a^tb\big)\equiv{\rm diag}
\big(c^td\big)\equiv0\ {\rm mod}\ 2n \right\rbrace,
\end{gather*}
first order theta constants are seen to transform as follows (cf.~\cite{Ig72} and~\cite{RF} for the general formula)
\[
\tt\ep\de(\gamma \cdot \tau,0) = \kappa(\gamma) \chi_{\ep,\de}(\gamma) \det{(c \tau + d)}^{\frac{1}{2}} \tt\ep\de(\tau,0) \qquad \forall\, \gamma \in \Gamma_g[2],
\]
where $\kappa(\gamma)$ is an eighth root of the unity for any $\gamma$ and $ \chi_{\ep,\de}$ are characters of the group $\Gamma_g[2,4]/\Gamma_g[4,8]$; in particular, an action of $\Gamma_g[2,4]/\Gamma_g[4,8]$ is naturally induced on theta constants and can be described in terms of these characters $ \chi_{\ep,\de}$, which span the character group of $\Gamma_g[2,4]/\pm\Gamma_g[4,8]$ (cf.~\cite{SM}). Thanks to the transformation formula, first order theta constants induce a map of $\HH_g/\pm\Gamma_g[4,8]$ into the projective space; this map is known to be an immersion, cf.~\cite{Ig72}. Second order theta constants also induce a map on a finite cover of the moduli space of ppavs
\[
\mathcal{T}\colon \ \HH_g /\Gamma_g[2,4] \to \PP^{2^g-1},
\]
which is generically injective for any $g$, cf.~\cite{SM}. Once a basis $\omega_1, \dots , \omega_g$ for the cohomology of an algebraic curve of genus $g$ is chosen together with a symplectic basis of cycles $\de_1, \dots , \de_g, \de'_1, \dots, \de'_g$, the $g \times 2g$ matrix $\big(\int_{\de_j}\omega_i, \int_{\de'_j}\omega_i\big)$ defines a complex torus which is isomorphic to a ppav $X_{\tau}$; this complex torus is known as the {\it Jacobian variety} of the curve and the locus in the moduli space of ppavs defined by those ppavs that are isomorphic to Jacobians of curves is known as the {\it Jacobian locus}. Because of the existence of the map $\mathcal{T}$, one knows that the Jacobian locus and the {\it hyperelliptic locus} (i.e., the locus defined by those ppavs that are isomorphic to Jacobians of hyperelliptic curves) can be described in terms of equations involving theta constants by taking the preimages of these loci under the covering map $\HH_g /\Gamma_g[2,4] \to \HH_g /{\rm Sp}(2g, \ZZ)$; the preimage of the hyperelliptic locus has several irreducible components in $\HH_g /\Gamma_g[2,4]$.

As for the hyperelliptic locus, we know from the result in~\cite{mu} that the irreducible components are defined in terms of vanishing and non-vanishing conditions for certain theta constants $\theta\left[\begin{smallmatrix}\ep \\ \de \end{smallmatrix}\right] (\tau,0)$. Mumford's result is based on Neumann's dynamical system. In~\cite{po} Poor used the cross ratio to prove that the vanishing conditions alone are sufficient to characterize the hyperellipitic locus, once the irreducibility of the principally polarized abelian variety is assumed. In both cases a fundamental tool is the so-called Frobenius theta formula, which is a consequence of Riemann's theta formula (cf.~\cite{FR} and~\cite{mu}). A~few years ago in \cite{Gr}, a characterization of the hyperelliptic abelian varieties was given in terms of cubic equations in the second order theta functions. These cubics were obtained by using the explicit coefficients for the addition formula from~\cite{BK2}.

The aim of this note is twofold: we first relate Grushevsky's approach to the others' by proving that his cubic equations are related to Frobenius' theta formula; we then give a new proof of Mumford's characterization by applying Gunning's multisecant formula (cf.~\cite{Gu}), which is a remarkable geometrical characterization of the locus of Jacobians in terms of intersections of g-dimensional linear subspaces with the Kummer variety of an abelian variety.

In this way we have an explicit relation between various approaches; in particular, we get an immediate link between the geometrical characterization and the explicit equations for the hyperelliptic locus.

Incidentally, we also prove that Grushevsky's cubics are induced by the quadratic relations related to the theta vanishing.

\section{Theta functions}

Throughout the rest of the paper we shall omit the subscript $\tau$ whenever it is clear that $\tau$ is fixed.
For the topics introduced here we shall follow Section 1 in \cite{vG} closely. Once we set $ M:= L^{\otimes 2}$, as $t_{x}^*M\cong M$ if and only if $x \in X[2]$, the following group is well defined
\[
G(M):=\big\{(x, \phi) \,|\, x\in X[2], \, \phi\colon t_{x}^*M \xrightarrow{\cong} M \big\}
\]
and an irreducible action of $G(M)$ on the space of global sections of $M$ is also defined by setting $(x, \phi) s := \phi (t_{x}^*s)$.

A theta structure is an isomorphism between $G(M)$
and the Heisenberg group, namely the set
\[
H:=\CC^{*}\times\ZZ_{2}^g\times {\rm Hom}(\ZZ_{2}, \CC^{*})^g,
\]
provided with the group law
\[
(t,x,x^*) \cdot (s,y,y^*) :=(tsy^*(x), x+y, x^*+y^*).
\]

Now, let $B_n:=\Gamma(X, M^n)$ for any $n \geq 1$.
The action of $\{ \pm 1 \}\subset {\rm End}( X)$ decomposes each of these spaces into the direct sum of two factors $B_n ^{+}$ and $ B_n ^-$. One has $B_1= B_1^+$ and $B_1$ is an irreducible representation of $H$.
Moreover, $B_1$ admits a basis $\{X_{\s}\}$ with $\s\in\ZZ_2^g$ such that
$(t, x, x^*)X_{\s}=tx^*(x+\s)X_{x+\s}$. The Heisenberg group naturally acts on the $n$-fold symmetric products $S^nB_1$ as well. To decompose~$S^nB_1$ we introduce the two following subgroups of $H$:
\[
K=\{(t, x, x^*)\colon t=1,\, x=0\} , \qquad K^*=\{(t, x, x^*)\colon t=1,\, x^*=0\},
\]
and we recall from \cite{vG} the following:

\begin{Proposition}\label{pr.vG1}\quad
\begin{enumerate}\itemsep=0pt
\item[$(i)$] A basis of eigenvectors for the action of H on $S^2B_1$ is given by the $2^{g-1}\big( 2^{g}+1\big)$ elements $Q[\ep, \ep']=\sum(-1)^{\langle \s,\ep'\rangle}X_{\s}X_{\s+\ep}$, where $\s, \ep, \ep'\in\ZZ_2^g$, the symbol $\langle \cdot,\cdot\rangle$ stands for the standard inner product and $\langle \ep,\ep'\rangle=0$. The action of $H$ on the elements of the basis is given by
\[
(t, x, x^*)Q[\ep, \ep']=t^2(-1)^{\langle x,\ep'\rangle}x^*(\ep)Q[\ep, \ep'].
\]
\item[$(ii)$] The space $S^3B_1$ is a direct sum of $\big(2^g+1\big)\big(2^{g-1}+1\big)/3$ copies of a~$2^g$-dimensional irreducible representation of~$H$.
\item[$(iii)$]Let $X(H)$ be the group of characters of $\ZZ_{2}^g\times {\rm Hom}(\ZZ_{2}, \CC^{*})^g$. Then
\[
S^4B_1= \bigoplus_{\chi\in X(H)} S^4_{\chi}(B_1).
\]
The dimension of the eigenspace associated with the trivial character is equal to $\big(2^g+1\big)\big(2^{g-1}+1\big)/3$, whereas the dimension of the other eigenspaces is $\big(2^{g-1}+1\big)\big(2^{g-2}+1\big)/3$.
\end{enumerate}
\end{Proposition}

Moreover, a basis for $S^3B_1$ as an $H$-module is given by a basis of the $K$-invariant elements and a basis for $S^4_{0}(B_1)$ is also described in \cite{vG}.

Maps between these spaces can be defined. Let $S^3_{0}(B_1)$ be the subspace of $K$-invariant elements of $S^3(B_1)$.
 For $F\in S^3_{0}(B_1)$ and $ \s\in \ZZ_{2}^g$ let
\[
F_{\s}:=( 1, \s, 0)F.
\]
For $\chi=(y^*, y)\in X(H)$ (so that $\chi(x, x^*)=y^*(x)x^*(y)$) we define
\begin{gather*}
M(\chi)\colon \ S^3_{0}(B_1)\to S^4_{\chi}(B_1),\\
M(\chi)(F)= \sum_{\s\in \ZZ_{2}^g} y^*(\s)X_{\s+y}F_{\s}.
\end{gather*}

 \begin{Proposition} \label{pr.vG2}The maps $M(\chi)$ are surjective. In the case $\chi=0$ the map $M(0)$ is an isomorphism and the inverse is given by $ 4^{-1}\frac{\partial}{\partial X_0}$. Moreover
\[
\left(\frac{\partial}{\partial X_{\s}}\right)(M(0)F)=4F_{\s},\qquad \forall\, \s\in \ZZ_{2}^g.
\]
\end{Proposition}

We will briefly discuss the case of the quartics; more details on the subject can be found in~\cite{vG}.

The decomposition of $S^4B_1$ is given in Proposition~\ref{pr.vG1}(iii); a similar decomposition holds for the space
\[
S^2\big(S^2B_1\big)=\bigoplus_{\chi} \big(S^2\big(S^2B_1\big)\big)_{\chi}.
\]

In this case the dimensions of the eigenspaces are respectively equal to $2^{g-1} \big(2^g+1\big)$ and $2^{g-2}\big(2^{g-1}+1\big)$. Obviously, we have a surjective map
\[
 \Phi_4\colon \ S^2\big(S^2B_1\big)\to S^4B_1.
\]
 According to \cite{fay} (cf.\ also~\cite{APS} or~\cite{FGS}), the relations are generated by a particular case of the so-called Riemann relations, namely those of the form
\begin{gather}\label{eq:RR}
\sum_{\langle\ep,\de\rangle =0} (-1)^{\langle \s, \de\rangle} v_{\ep,\de}Q[\ep, \de]Q[\ep+\s, \de+\r] .
\end{gather}

Here the vector $v=( v_{\ep,\de})$ varies in a suitable space of dimension $\big(2^{2g}-1\big)/3$. We say that these are biquadratic relations. By evaluating at the theta functions
$X_{\s}=\T[\s](\tau,z)$ and applying the addition formula, we get
\begin{gather}\label{eq:vT}
Q[\ep, \ep'](\dots,\T[\s](\tau,z),\dots):=Q[\ep, \ep'](\tau, z) =\tt \ep {\ep'} (\tau, 0)\tt \ep {\ep'} (\tau, 2z).
 \end{gather}
Hence, a basis for the quadrics containing the Kummer variety is given by those $Q[\ep, \ep']$ such that
$\theta\left[\begin{smallmatrix}\ep \\ \ep' \end{smallmatrix}\right] (\tau, 0)=0$.

To sum up, we can also state the following:

\begin{Proposition} \label{trivial relations}
All biquadratic Riemann relations induce trivial relations between second order theta functions.
\end{Proposition}

 We know there also exist highly non-trivial relations like the equation defining the Kummer surface in genus 2 or the Coble quartic in genus~3. These are in the kernel of the map
\[
\Psi_4\colon \ S^4B_1\to B_4^+.
\]
\begin{Remark} More general Riemann relations are obtained by evaluating (\ref{eq:RR}) at the theta functions and then at the points $z+(\tau x+y)/4$, $x,y \in\ZZ^g$, which yields
\begin{gather}\label{eq:RR2}
\sum_{\langle\ep,\de\rangle =0}\! (-1)^{\langle\s+x, \de\rangle } v_{\ep,\de} \tt\ep\de(\tau,0)\tt{\ep+\s }{\de+\r}(\tau,0)\tt{\ep+x}{\de+y}(\tau,2z)\tt{\ep+\s+x}{\de+\r+y}(\tau,2z)=0.\!\!\!
\end{gather}
\end{Remark}

\section{Cubic equations}

As soon as the genus is greater than 2, there are also cubic equations defining the Kummer va\-riety. To focus on cubic relations, we first note that the spaces $S^3 B_1$ and $B_3^+$ decompose under the action of the Heisenberg group into $2^g$ irreducible representations that are all isomorphic to~$B_1$ (see~\cite{vG} or \cite{Bea}); we are interested in studying the kernel of the map $ \Psi_3\colon S^3B_1\to B_3^+$. A~simple way of constructing cubic relations is considering quadratic relations. Quadratic relations exist whenever we have the vanishing of some even theta constants i.e., $ \theta\left[\begin{smallmatrix}\ep \\ \de\end{smallmatrix}\right] (\tau,0)=0$ with $\langle \ep , \de\rangle =0$ ${\rm mod}\, 2$. If none of the theta constants vanish, we can prove there always exist cubic relations that are not a product of a quadratic and a linear form; this is done by means of a dimensional argument, as we have
\[
 \dim_{\CC} S^3 B_1=2^g\big(2^g+1\big)\big(2^{g-1}+1\big)/3> 2^{g-1}\big( 3^g+1\big)=\dim_{\CC} S^3 B_3^+.
\]
 In the genus 3 case there are exactly 8 cubics: in the non-hyperelliptic case (no vanishing theta constants) these cubics can be obtained as derivatives of the Coble quartic, while in the hyperelliptic case (one vanishing theta constant) they are obtained as a product of a quadric $Q[\ep, \ep'](\tau, z)$ and a linear form in the $\T[\s](\tau,z)$.

When the genus is higher, we need to determine a set of theta constants $\theta \left[\begin{smallmatrix} \ep \\ \ep ' \end{smallmatrix}\right] (\tau,0)$ that vanish at the point $\tau$ in the hyperelliptic locus; to do this, we start with a hyperelliptic point to which we can associate a special fundamental system of characteristics, namely an azygetic collection of $2g+2$ characteristics $m_1, \dots, m_g, m_{g+1}, \dots, m_{2g+2}$ such that the first $g$ are odd and the last $g+2$ are even. Once we denote by $\{e_k\}_{k=1, \dots, g}$ the elements of the natural basis in ${\ZZ^g_2}$, and set $e_{g+1}= 0$ and $s_k= e_1+\dots+e_k$, we can choose the following as special fundamental system
\begin{gather}
 m_1:=\begin{bmatrix} s_1 \\ e_1\end{bmatrix},\qquad \dots, \qquad m_k:=\begin{bmatrix} s_k \\ e_k\end{bmatrix},\qquad\dots,\qquad m_g:=\begin{bmatrix} s_g \\ e_g\end{bmatrix},
\nonumber\\
m_{g+1}:=\begin{bmatrix} 0 \\ e_{1}\end{bmatrix},\qquad \dots,\qquad m_{2g}:=\begin{bmatrix} s_{g-1} \\ e_{g}\end{bmatrix}, \qquad m_{2g+1}:=\begin{bmatrix} s_g \\ 0
\end{bmatrix},\qquad m_{\infty}:=\begin{bmatrix} 0 \\ 0 \end{bmatrix} .\label{eq:FS}
\end{gather}

To be consistent with Mumford's notation in \cite[p.~106]{mu}, we set
\begin{gather}\label{BU}
B:=\{1, \dots, 2g+ 1, \infty \} ,\qquad U:= \{ g+1, \dots , 2g+1 \},
\end{gather}
and for any subset $S\subset B$ we denote by $ CS$ the complementary set in $B$.
 To an even subset $T\subset B$ (modulo $S \sim CS$) we associate a characteristic according to the rule
 \[
T\to m_T= \sum_{j\in T} m_j .
\]
This actually defines a bijection.

By setting $T\circ S:= (T\setminus S)\cup (S\setminus T)$, it can be checked that the parity of $m_T$ is equal to $(-1)^{\frac{\#(T\circ U)-g-1}{2}}$.

We then know from the characterization of the hyperelliptic locus (cf.~\cite{Th} and~\cite{mu}) that the period matrix of a hyperelliptic curve always admits a conjugate $\tau \in \HH_g$ under the action of the group ${\rm Sp}(2g, \ZZ)$, for which the following vanishing and non-vanishing conditions for theta constants hold
\begin{gather}
\theta_{m_T} (\tau,0)=0 \qquad \text{if} \quad \#(T\circ U)\neq g+1, \label{vanishing}\\
\theta_{m_T} (\tau,0) \neq 0 \qquad \text{if} \quad \#(T\circ U) = g+1. \label{nonvanishing}
\end{gather}

{\it The case $g=4$}. To obtain all the even characteristics $m_T$ corresponding to vanishing theta constants we have to choose $T$ such that $\#T=4$ and
\[
T\subset\{ 1, 2, 3, 4, \infty \} \qquad {\rm or}\qquad
 T\subset\{ 5, 6, 7, 8, 9 \}.
\]
Thus, we get the following characteristics
\begin{gather}
\left[\begin{matrix}\ep_0\\ \de_0\end{matrix}\right]=\begin{bmatrix} 0 1 0 1 \\ 1 1 1 1 \end{bmatrix},\qquad \left[\begin{matrix}\ep_1\\ \de_1\end{matrix}\right]=\begin{bmatrix}1 1 0 1 \\ 0 1 1 1 \end{bmatrix}, \qquad \left[\begin{matrix}\ep_2 \\ \de_2 \end{matrix}\right]=\begin{bmatrix}1 0 0 1 \\ 1 0 1 1 \end{bmatrix}, \qquad \left[\begin{matrix}\ep_3 \\ \de_3 \end{matrix}\right]=\begin{bmatrix}1 0 1 1 \\ 1 1 0 1 \end{bmatrix}, \nonumber\\
 \left[\begin{matrix}\ep_4 \\ \de_4 \end{matrix}\right]=\begin{bmatrix} 1 0 1 0 \\ 1 1 1 0 \end{bmatrix},\qquad
\left[\begin{matrix}\ep_5 \\ \de_5 \end{matrix}\right]=\begin{bmatrix} 0 1 0 1 \\ 0 1 1 1 \end{bmatrix},\qquad \left[\begin{matrix}\ep_6 \\ \de_6 \end{matrix}\right]=\begin{bmatrix} 1 1 0 1 \\ 1 0 1 1 \end{bmatrix},\qquad \left[\begin{matrix}\ep_7 \\ \de_7 \end{matrix}\right]=\begin{bmatrix} 1 0 0 1 \\ 1 1 0 1 \end{bmatrix}, \nonumber\\
\left[\begin{matrix}\ep_8 \\ \de_8 \end{matrix}\right]=\begin{bmatrix} 1 0 1 1 \\ 1 1 1 0 \end{bmatrix}, \qquad \left[\begin{matrix}\ep_9 \\ \de_ 9 \end{matrix}\right]=\begin{bmatrix} 1 0 1 0 \\ 1 1 1 1 \end{bmatrix}.\label{10vanishing}
\end{gather}

Note that in \cite{Ig} the special fundamental system is given by a set of characteristics which is obtained from the above collection by switching the $\ep$ with the $\de$. Nevertheless, we prefer to use our form, since it is compatible with the notation in \cite{Gr}.

For this collection of vanishing theta constants the 10 induced quadratic relations $Q [ \epsilon , \epsilon ' ]=0$ produce 160 generators for $\operatorname{Ker} \Psi_3$ that are linearly dependent, as the dimension of the space is $144$ (computed with Mathematica).
Thus, there are (at least) $160 -144 = 16$ other cubic relations.

Another known example of a point $\tau \in \HH_4$ with a set of $10$ vanishing even theta constants is given by the so-called Varley--Debarre abelian variety, which is uniquely determined modulo the action of the integral symplectic group. For a suitable period matrix associated with this variety the set of characteristics corresponding to vanishing theta constants can be deduced from~\cite{vG2} or~\cite{Be} and chosen as follows
\begin{gather*}
n_1=\begin{bmatrix} 0 0 0 0 \\ 0 0 0 0 \end{bmatrix}, \qquad n_2=\begin{bmatrix} 1 0 1 0 \\ 0 0 0 0 \end{bmatrix}, \qquad n_3=\begin{bmatrix} 0 1 0 1 \\ 0 0 0 0 \end{bmatrix}, \qquad n_4=\begin{bmatrix} 1 1 1 1 \\ 0 0 0 0 \end{bmatrix}, \qquad n_5= \begin{bmatrix} 0 1 0 1 \\ 1 0 1 0 \end{bmatrix}, \\
n_6=\begin{bmatrix} 1 0 1 0 \\ 0 1 0 1 \end{bmatrix}, \qquad n_7=\begin{bmatrix} 0 0 0 0 \\ 1 0 1 0 \end{bmatrix}, \qquad n_8=\begin{bmatrix} 0 0 0 0 \\ 0 1 0 1 \end{bmatrix}, \qquad n_9=\begin{bmatrix} 0 0 0 0 \\ 1 1 1 1 \end{bmatrix}, \qquad n_{10}= \begin{bmatrix} 1 1 1 1 \\ 1 1 1 1 \end{bmatrix}.
\end{gather*}

These $10$ characteristics can be determined as follows: they belong to a coset of a $4$-dimen\-sio\-nal subspace of $\ZZ_2^8$ that contains all even elements and they are determined by the unique condition that the complementary set $\{n_{11},\dots, n_{16}\}$ satisfies
\[
n_{11}+\dots +n_{16}=0.
\]
This condition can be deduced from Proposition~4.5 in~\cite{vG2}, where the $2$-torsion points
 $p_1$, $p_2$, $p_3$, $p_4$, $p_1+p_3$, $p_2+p_4$ are given, or from corollary of Proposition~3 in~\cite{Be}. Here we can parametrize the set of quadratic forms associated with the symplectic form with characteristics $m$ of the type $m_i \oplus m_i$ with $m_i\in \ZZ_2^4$. Since the condition on the Arf invariant is related to the parity of the characteristics~$m_i$, we get the result.

For such a collection the dimension of the kernel of $\Psi_3$ turns out to be equal to $160$. This can be seen straightforwardly on the set of the ten even characteristics $[\ep, \alpha]$ such that $\ep = 0$ and $\alpha$ has exactly two or three entries equal to $1$; as this set also satisfies the above condition, it can actually be associated with the vanishing theta constants of a point corresponding to the Varley--Debarre variety up to conjugates; the $160$ cubics induced in this case are $(X[\r] Q[0, \alpha])(\tau,z)= X[\r] \cdot \sum\limits_{\s \in \ZZ_2^{4}}(-1)^{\langle \alpha, \s\rangle} X^2[\s]$ with $\alpha$ as before and $\r \in \ZZ_2^{4}$, which are easily seen to be linearly independent. Hence, we will not have new cubics if the map $ \Psi_3\colon S^3B_1\to B_3^+$ is surjective.

In general, we expect to find cubic equations in the hyperelliptic case that are not generated by the quadrics. Those given in \cite{Gr}, however, turn out to be dependent from the quadrics; we will discuss these cases in detail and deduce some results.

First of all, we recall the following statement from \cite{Gr}:

\begin{Theorem}\label{th.Gr} An irreducible period matrix $\tau\in \HH_g$ is the period matrix
of a hyperelliptic Jacobian with the basis of cycles chosen in such a way that the corresponding special fundamental system is the one given in~\eqref{eq:FS} if and only if the following cubic identities for second order theta functions are satisfied for all $\s\in \ZZ_2^g$ and for all $z\in \CC^g$:
 \begin{gather}\label{eq:Cu1}
 R_{\s}:= Q[0,0](\tau, z)\T[\s](\tau,z)- \sum_{k=0}^g (-1)^{\langle \s, e_{k+1}\rangle}
Q[s_k, e_{k+1}](\tau, z)\T[\s+s_k](\tau,z)=0,
\end{gather}
where we assume that not all the terms $ Q[0,0](\tau, z)$ and $Q[s_k, e_{k+1}](\tau, z)$ are identically zero in~$z$.
\end{Theorem}
\begin{Remark} The $2^g$ cubic polynomials appearing in (\ref{eq:Cu1}) span an irreducible representation of the Heisenberg group and are of the form
\[R_{\s}=(1,\s,0)R_0.
\]
\end{Remark}

We recall that the result $R_{\s}=0$ is obtained by computing the coefficients in the
formula given in \cite{BK2}.

Indeed, in \cite{Gr} the following relations appear
\begin{gather}
 \sum_{\ep\in\ZZ_2^g} \T[\ep](\tau,z)\T[\ep](\tau,z)\T[\s](\tau,z)\nonumber\\
\qquad {}= \sum_{k=0}^g \sum_{\ep\in\ZZ_2^g} (-1)^{\langle \ep+\s, e_{k+1}\rangle}\T[\ep](\tau,z)\T[\ep+ s_k](\tau,z)\T[\s+s_k](\tau,z). \label{eq:Cu}
\end{gather}

It is an immediate consequence of (\ref{eq:vT}) and the definition of $ Q[\ep,\de](\tau, z)$, that (\ref{eq:Cu1}) and~(\ref{eq:Cu}) are the same equations.

The condition on the non-identically zero terms in Theorem~\ref{th.Gr} turns into the assumption that not all $\theta\left[\begin{smallmatrix}\ep \\ \de\end{smallmatrix}\right](\tau,0) $ are null for
 \[
 \begin{bmatrix} \ep \\ \de\end{bmatrix} =\begin{bmatrix} 0 \\ 0 \end{bmatrix} ,\begin{bmatrix} 0 \\ e_{1}\end{bmatrix}, \dots\begin{bmatrix} s_k \\ e_{k+1} \end{bmatrix},\dots,\begin{bmatrix} s_g \\ 0
\end{bmatrix}.
\]
We observe that this assumption is required in Mumford's characterization of the hyperelliptic locus in \cite{mu}. Actually, the non-vanishing assumption (\ref{nonvanishing}) is stronger.

As proved in \cite{Gr}, these equations in genus 3 are
equivalent to the vanishing of one theta constant, yet when it comes to the genus~4 case they appear to be new; however, as already said, this is not the case, as we will explain. First of all, we need to recall from \cite[Theorem~7.1]{mu}, the so-called generalized Frobenius theta formula. For $U\subset B$ as in (\ref{BU}) we set $\ep_U(j):= \pm1$ according as $j\in U$ or $j\notin U$; then we have

\begin{Theorem} Let $\tau\in\HH_g$ satisfy the vanishing conditions in~\eqref{vanishing}. Then $\forall\, z_1, z_2, z_3, z_4 \in\CC^g$ such that $ \sum\limits_{i=1}^4 z_i=0$:
\begin{gather}\label{eq:FR} \quad \sum_{j\in B} \ep_U(j)\prod_{i=1}^4\theta[m_j](\tau, z_i)=0 . \end{gather}
\end{Theorem}

In \cite{mu} extra variables $a_i$ actually appear, but, as shown in \cite{po}, they are redundant.

We remark that if we apply the addition formula for theta functions to Grushevsky's cubic equations, we obtain similar equations that are also consequence of the generalized Frobenius theta formula; in fact, we have the following:

\begin{Proposition}\label{statement:GF'} If $\tau\in \HH_g$ satisfies equations \eqref{eq:Cu1},
 then, for all $z\in \CC^g$:
\begin{gather}\label{eq:RF2a}
 Q[0, 0] (\tau, z)Q[\ep_1, \de_1 ] (\tau, z)=\sum_{k=0}^g Q[s_k, e_{k+1}] (\tau, z)Q[s_k+\ep_1, e_{k+1}+\de_1] (\tau, z) .
\end{gather}
\end{Proposition}

\begin{proof} We only need to consider
\begin{gather*}
Q[0,0](\tau, z)\sum_{ \s} (-1)^{\langle \s,\de_1\rangle}\T[\s](\tau,z)\T[\s+\ep_1](\tau,z)\\
\qquad{}= \sum_{k=0}^g Q[s_k, e_{k+1}](\tau, z)\sum_{ \s} (-1)^{\langle \s,\de_1+e_{k+1}\rangle}
\T[\s+s_k](\tau,z)\T[\s+\ep_1](\tau,z).
\end{gather*}

Now, the l.h.s.\ yields $ Q[0, 0](\tau, z)Q[\ep_1, \de_1 ] (\tau, z)$. Moreover, we have
\[
Q[s_k+\ep_1, e_{k+1}+\de_1](\tau, z)= \sum_{\s \in\ZZ_2^g} (-1)^{\langle\s, e_{k+1} +\de_1\rangle} \T[\s+s_k+\ep_1](\tau,z)\T[\s](\tau,z).
\]

Hence, in the r.h.s.\ we obtain
\[
\sum_{k=0}^g Q[s_k, e_{k+1}] (\tau, z)Q[s_k+\ep_1, e_{k+1}+\de_1] (\tau, z).\tag*{\qed}
\]\renewcommand{\qed}{}
\end{proof}

\begin{Remark} As a referee correctly pointed out, in the above proposition we are computing $M(\chi)R_0$. \end{Remark}

As a particular yet fundamental case, we will consider $M(0)R_0$, i.e., the case obtained by setting $\ep_1=\de_1=0$; hence we get

\begin{Corollary}\label{statement:GF} If $\tau\in \HH_g$ satisfies equations \eqref{eq:Cu1}, for all $z\in \CC^g$ we have
\begin{gather}\label{eq:RF}
 Q[0, 0]^2(\tau, z)=\sum_{k=0}^g Q[s_k, e_{k+1}]^2(\tau, z).
\end{gather}
\end{Corollary}

\begin{Remark} The important fact is that this formula is exactly the one given in \cite[Corollary~7.5, p.~113]{mu} when $S=\varnothing$, which is a particular case of the general formula obtained as a~consequence of the vanishing conditions~(\ref{vanishing}).
 \end{Remark}

Formula (\ref{eq:RF}) is therefore a special case of Frobenius' theta formula and it is fundamental in Mumford's investigation of Neumann's dynamical system (see \cite[Lemma~9.7]{mu}). It is worth recalling that this formula is obtained by plugging the vanishing conditions in a particular biquadratic Riemann relation of the type (\ref{eq:RR}) with $v_{\ep,\de}=\pm 1$ and $\s=\rho=0$; in fact, we can get it from (\ref{eq:FR}) by setting $z_1=z_2=0$, $z_3=z$, $z_4=-z$ and $ m_1=m_2=m_3=m_4.$

To obtain all of Frobenius' relations described by Mumford in the above cited Corollary, we can evaluate equation (\ref{eq:RF}) at the points $z+(\tau x+y)/4$ for $x,y \in\ZZ^g$, in the same way we did with Riemann's relations to obtain (\ref{eq:RR2}); this leads to the following
\begin{Corollary} If $\tau\in \HH_g$ is a period matrix
satisfying equation \eqref{eq:RF}, then
\begin{gather}\label{eq:RF2}
 \tt{0}{0}(\tau,0)^2\tt{x}{y }(\tau, 2z)^2 =\sum_{k=0}^g \tt{s_k}{e_{k+1}}(\tau,0)^2\tt{s_k+x}{e_{k+1}+y }(\tau, 2z)^2 .
\end{gather}
\end{Corollary}

The converse of Corollary~\ref{statement:GF} also holds:

 \begin{Proposition}\label{statement:FG} Let $\tau\in\HH_g$ satisfy the vanishing conditions in~\eqref{vanishing}; then equation \eqref{eq:RF} implies equations~\eqref{eq:Cu1}.
 \end{Proposition}
 \begin{proof}We know that equation (\ref{eq:RF}) is obtained by evaluating a suitable biquadratic Riemann relation of the form
 \[
\sum \pm Q[\ep, \de]^2 =0.
\]
But this is a trivial relation among the $X_{\s}$, hence, considering the inverse of the map $M(0)$ (cf.\ Proposition~\ref{pr.vG2}) we deduce that the derivative with respect to $X_0$ also gives a trivial relation in the $X_{\s}$ and the same holds for the derivative with respect to any of the $X_{\s}$. Evaluating at the $\T[\s](\tau, z)$ we get the non-trivial relations $R_{\s}=0$.
\end{proof}

Another interesting consequence can be derived. As we recalled, Riemann's relations of the form
\begin{gather*}
\sum_{\langle\ep,\de\rangle =0} v_{\ep,\de} Q[\ep, \de]^2 =0
\end{gather*}
induce trivial quartic relations among the $X_{\s}$ (see Proposition~\ref{trivial relations}). We can take the suitable equation inducing Frobenius' formula and write it as
\begin{gather*}
\sum_{\langle\ep,\de\rangle =0/ \theta\left[\begin{smallmatrix}\ep \\ \de\end{smallmatrix}\right](\tau,0)\neq 0 } v_{\ep,\de}Q[\ep, \de]^2 +
 \sum_{\langle\ep,\de\rangle =0/ \theta\left[\begin{smallmatrix}\ep \\ \de\end{smallmatrix}\right](\tau,0)=0 } v_{\ep,\de}Q[\ep, \de]^2 =0,
\end{gather*}
hence
\begin{gather*}
\sum_{\langle\ep,\de\rangle =0/ \theta\left[\begin{smallmatrix}\ep \\ \de\end{smallmatrix}\right](\tau,0)\neq 0 } v_{\ep,\de}Q[\ep, \de]^2 =
 -\sum_{\langle\ep,\de\rangle =0/ \theta\left[\begin{smallmatrix}\ep \\ \de\end{smallmatrix}\right]\ep\de(\tau,0)=0 } v_{\ep,\de}Q[\ep, \de]^2.
\end{gather*}

We can now take the derivatives of both sides with respect to any of the $X_{\s}$ and evaluate them at the $\T[\s](\tau,z)$; we get Grushevsky's relations on the left side and a linear combination of terms that are products of a quadratic and a linear form in the $\T[\s](\tau,z)$ on the right side. This means Grushevsky's relations are contained in the ideal generated by those quadrics
 $Q[\ep, \de] (\tau, z)$ such that $\theta\left[\begin{smallmatrix}\ep \\ \de\end{smallmatrix}\right](\tau,0)=0$, hence we have the following:

\begin{Corollary} Grushevsky's cubics lie in the space spanned by linear forms and the above vanishing quadratic forms.
 \end{Corollary}

Thus, in the hyperelliptic case we have more cubic relations in genus $g=4$; we expect this to be true for any $g\geq 4$.

One is led to suppose that Frobenius' formula might imply the vanishing of the suitable theta constants. Let us discuss this formula in low genera. In the genus 3 case let us assume $\theta\left[\begin{smallmatrix}0 \\ 0 \end{smallmatrix}\right](\tau,0)\neq 0$, and evaluate formula (\ref{eq:RF2}) at $x=(1,0,1)^{\rm t}$,
$y=(1,1,1)^{\rm t}$ and $z=0$. Because of the vanishing of the theta constants with odd characteristics we obtain
\[
\tt{0}{0}(\tau,0)^2\tt{x}{y }(\tau, 0)^2 =0,
\]
hence $\theta\left[\begin{smallmatrix}1&0&1\\ 1&1&1 \end{smallmatrix}\right](\tau, 0)=0$.

It is a well-known fact that this equation characterizes a component of the hyperelliptic locus in genus~3.

In genus 4 we can use the same method or equation (\ref{eq:RF2a}) evaluated at the ten characteristics described in (\ref{10vanishing}), i.e., $x=\ep_i$ and $y=\de_i$. Again, we assume $\theta\left[\begin{smallmatrix}0 \\ 0 \end{smallmatrix}\right](\tau,0)\neq 0$, and use formula~(\ref{eq:RF2a}); we get
\begin{gather}\label{eq:v4}
 \tt{\ep_i}{\de_i}(\tau,0) =0, \qquad i=1,2,3,4,
\end{gather}
and six more equations, namely
\begin{gather*}
\tt{0}{0}(\tau,0)^2\tt{\ep_0}{\de_0 }(\tau, 0)^2 -\sum_{k=0}^4 \tt{s_k}{e_{k+1}}(\tau,0)^2\tt{\ep_{k+5}}{\de_{k+5} }(\tau, 0)^2=0,\\
\tt{s_k}{e_{k+1}}(\tau,0)^2\tt{\ep_0}{\de_0 }(\tau, 0)^2-\tt{0}{0}(\tau,0)^2\tt{\ep_{k+5}}{\de_{k+5} }(\tau, 0)^2=0, \qquad k=0,1,2,3,4.
\end{gather*}

Thus, we are left with a system of 6 equations in the six variables
$\theta\left[\begin{smallmatrix}\ep_{k} \\ \de_{k} \end{smallmatrix}\right] (\tau, 0)^2$ for $ k=0, 5, 6, 7, 8, 9$. The matrix of the coefficients is
\begin{gather*}
A= \left(\begin{matrix}
 \tt{0}{0}^2&-\tt{0}{e_{1}}^2 & -\tt{e_{1}}{e_{2}}^2 &- \tt{s_{2}}{e_{3}}^2 & -\tt{s_{3}}{e_{4}}^2&- \tt{s_{4}}{0}^2 \vspace{1mm}\\
 -\tt{0}{e_{1}}^2 & \tt{0}{0}^2 & 0 & 0& 0 & 0\vspace{1mm}\\
-\tt{e_{1}}{e_{2}}^2 &0& \tt{0}{0}^2 & 0 & 0& 0\\
\vdots&\vdots&\vdots&\ddots& &\vdots\\
\vdots&\vdots&\vdots& &\ddots&\vdots\\
- \tt{s_{4}}{0}^2& 0 & 0& 0 & 0& \tt{0}{0}^2 \\
 \end{matrix}\right),
\end{gather*}
which unfortunately has rank $5$, since the determinant is
\[
\tt{0}{0}^8\left( Q[0, 0]^2(\tau, 0)-\sum_{k=0}^4 Q[s_k, e_{k+1}]^2(\tau,0) \right)=0.
\]

Hence, a priori, the system can admit a non-trivial solution; one expects this solution to be ruled out by using Frobenius' formula in its full generality, i.e., the equation in~(\ref{eq:FR}).

As we are dealing with the genus~4 case, we actually know that the vanishing obtained in~(\ref{eq:v4}) means the point is hyperelliptic; Max Noether actually proved that the vanishing of three azygetic theta constants implies the vanishing of seven other forming an azygetic 10-tuple with them, although this $10$-tuple is not uniquely determined by the three (cf.~\cite[Section~14]{No}; in~\cite{No} the vanishing of these ten theta constants is said to imply that the period matrix is hyperelliptic via a result by Weierstrass.

A complete proof that the vanishing of four azygetic theta constants implies that the point is hyperelliptic can be also found in \cite{Igu}; we have to mention that Igusa's proof makes use of Riemann's relations of the form described in (\ref{eq:RR2}).

\section{Characterization of the Hyperelliptic locus}
Whenever $\tau$ is the period matrix of a hyperelliptic curve, a classical result (cf.~\cite{Th}) states that a~system of characteristics as in (\ref{eq:FS}) with the vanishing and non-vanishing properties~(\ref{vanishing}) and~(\ref{nonvanishing}) can be associated with~$\tau$ or with a conjugate of $\tau$ via the action of ${\rm Sp}(2g, \ZZ)$.

 It was not until 1984 that Mumford proved the converse statement in~\cite{mu}. We will now give a new proof of Mumford's result with a different approach that has the merit of involving Gunning's multisecant formula, as expressed in the following statement (cf.~\cite{Gu}):
 \begin{Theorem} Let $X$ be an irreducible principally polarized abelian variety of dimension~$g$, and let $A_0, \dots , A_{g+1}$ be distinct points of $X$. Suppose that $\forall\, z\in X$ the $g + 2 $ points ${\rm Th}_2(A_i + z) $ are linearly dependent.
Assume moreover the following general position condition: that there
exist some $k$ and $l$ such that for $y=(-A_k+A_l)/2$ the linear span of the points ${\rm Th}_2(A_i + y) $ for $ i = 0 , \dots , g+1$ has dimension precisely equal to $g+1$, and
not less. Then $X$ is the Jacobian of some curve $C$, and all the points $A_i$ are in the image of the Abel--Jacobi map.
\end{Theorem}

We are now in a position to prove that the above conditions are satisfied when cases~(\ref{vanishing}) and (\ref{nonvanishing}) hold; indeed, we have the following:

\begin{Theorem} Let $\tau$ be a point whose vanishing and non-vanishing theta constants are those in~\eqref{vanishing} and~\eqref{nonvanishing}; then $\tau$ is the period
matrix of a hyperelliptic Jacobian.
\end{Theorem}

\begin{proof}The period matrix turns out to be irreducible by counting the vanishing and non-vanishing theta constants or by checking the conditions of Theorem~4 in~\cite{SM}, i.e., the characters associated with the non-vanishing theta constants span the character group of
 $\Gamma_g[2,4]/\pm \Gamma_g[4,8]$.

Moreover, the vanishing conditions imply that Frobenius' formula (\ref{eq:RF})
\[
 Q[0, 0]^2 (\tau, z)=\sum_{k=0}^g Q[s_k, e_{k+1}]^2(\tau, z)
\]
 holds, hence Grushevsky's equations
 (\ref{eq:Cu1}) also hold with the coefficients being not all zero (see Proposition \ref{statement:FG}).

 We will prove that this implies the vectors ${\rm Th}_2(A_i + z) $ are linearly dependent.
 Here we take
 \[
A_0=0, \ A_1=(\tau s_0 + e_1)/2, \ \dots, \ A_k= (\tau s_{k-1} + e_k)/2, \ \dots, \ A_{g+1}= (\tau s_{g} + e_{g+1})/2,
\]
 with the corresponding characteristics being
\[
\begin{bmatrix} \ep \\ \de\end{bmatrix} =\begin{bmatrix} 0 \\ 0 \end{bmatrix} ,\begin{bmatrix} s_0 \\ e_{1}\end{bmatrix}, \dots\begin{bmatrix} s_k \\ e_{k+1} \end{bmatrix},\dots,\begin{bmatrix} s_g \\ e_{g+1} \end{bmatrix}.
\]

The $\s$-th coefficient of the vector ${\rm Th}_2(A_i + z)$ is
\[
\T[\s](\tau, z+(\tau s_{i-1} + e_i)/2)= (-1)^{\langle \s, e_i\rangle}\T[\s+s_{i-1}](\tau, z).
\]
Thus, the $2^g$ polynomials $R_\s$ in~(\ref{eq:Cu1}) give one non-trivial linear relation between the ${\rm Th}_2(A_i + z)$ for any~$z$.

We are yet to prove the general position condition. We set
\[
y=(A_1-A_0)/2= (1/4,0,\dots , 0)^{\rm t},
\]
and collect the points ${\rm Th}_2(A_i+y)$ in a $(g+2)\times 2^g$ matrix
\[
A=
 \begin{array}{ccccc}
 \begin{bmatrix} \T[0](A_0+y)&\dots&\dots& \T[\ep] (A_0+y)&\dots\\
\T[0](A_1+y)&\dots&\dots& \T[\ep] (A_1+y)&\dots\\
\dots&\dots&\dots&\dots&\dots\\
\T[0](A_{g+1}+y)&\dots&\dots& \T[\ep] (A_{g+1}+y)&\dots
 \end{bmatrix}.
 \end{array}
\]

 Now $A A^{\rm t}= B=(b_{ij}) $ with
\[
b_{ij}=\sum_{\ep} \T[\ep] (A_i+y)\T[\ep] (A_j+y)=\tt {0}{0}(A_i-A_j)\tt {0}{0}(A_i+A_j+2y).
\]

Up to constants we have
\[
\tt {0}{0}(A_0-A_j)=\tt{s_j}{e_{j+1}} \qquad {\rm and}\qquad \tt {0}{0}(A_0+A_j+2y)=\tt{s_j}{e_{j+1}+e_1}
\]
and for $1\leq i\leq j\leq g+1$
\[
\tt {0}{0}(A_i-A_j)=\tt{s_i+s_j}{e_{i+1}+e_{j+1}} \qquad {\rm and}\qquad \tt {0}{0}(A_i+A_j+2y)=\tt{s_i+s_j}{e_{i+1}+e_{j+1}+e_1}.
\]

Thus, the matrix $B$ turns into
\[ B=
 \begin{bmatrix} \tt {0}{0} \tt {0}{e_1} & \tt {0}{0} \tt {0}{e_1}&0&0&\dots&0\vspace{1mm}\\
\tt {0}{0} \tt {0}{e_1} & \tt {0}{0} \tt {0}{e_1}&0&0&\dots&0\\
0& 0&\tt {0}{0} \tt {0}{e_1}&0&\dots&0\\
\vdots & \vdots &\vdots &\ddots& &\vdots\\
0& 0&0&\dots& &\tt {0}{0} \tt {0}{e_1}
 \end{bmatrix}.
\]
Therefore, the matrix $B$ has rank $g+1$, hence $g+1\geq \operatorname{rk}(A)\geq \operatorname{rk}(B)=g+1$, which means the matrix $A$ has the required rank. Then $\tau$ is the period matrix of a Jacobian. As explained in~\cite{Gr}, since we are using points of order~2, we get that $2A_i- 2A_j=0$ as points of the Jacobian of the curve. By Abel's theorem this means there exists a function on the curve whose divisor is equal to $2A_i -2A_j$, i.e., the curve is hyperelliptic.
\end{proof}

\subsection*{Acknowledgements}
The authors would like to thank Bert van Geemen for drawing their attention to the result in~\cite{vG2}. They are also grateful to Sam Grushevsky for many helpful discussions and explanations. The authors are greatly indebted to an anonymous referee for the careful reading and suggestions.

\pdfbookmark[1]{References}{ref}
\LastPageEnding

\end{document}